\documentclass{amsart}
\usepackage{amsfonts}
\usepackage{graphicx}
\usepackage{amscd}
\usepackage{amsmath}
\usepackage{amssymb}
\usepackage{color}

\makeatletter
\@namedef{subjclassname@2010}{%
  \textup{2010} Mathematics Subject Classification}
\makeatother

\theoremstyle{plain}

\newtheorem{corollary}{\bf Corollary}
\newtheorem{definition}{\bf Definition}
\newtheorem{lemma}{\bf Lemma}

\newtheorem{proposition}{\bf Proposition}
\newtheorem{remark}{Remark}

\newtheorem{theorem}{\bf Theorem}

\theoremstyle{definition}

\numberwithin{equation}{section}

\begin{document}
\title[Gradient Ricci solitons]{Gradient Ricci solitons carrying a closed conformal vector field}

\author{J. F. Silva Filho}
\author{R. Sharma*}\thanks{*Corresponding author}

\address[J. F. Silva Filho]{Instituto de Ci\^encias Exatas e da Natureza, 
Universidade da Integra\c c\~ao Internacional da Lusofonia Afro-Brasileira - 
UNILAB,  62.790-000, Reden\c c\~ao - CE, Brazil}
\email{joaofilho@unilab.edu.br}

\address[R. Sharma]{Department Of Mathematics, University Of New Haven, West Haven,
CT 06516, USA}
\email{rsharma@newhaven.edu}

\date{\today}

\begin{abstract}
We show that a complete gradient Ricci soliton $(M^n,\,g)$ with constant scalar curvature and a  non-parallel closed conformal vector field is isometric to either the Euclidean space, or an Euclidean sphere, or negatively Einstein warped product of the real line with a complete  non-positively Einstein manifold. Moreover, we show that a  K\"ahler gradient Ricci soliton of real dimension $\ge 4$, with a non-parallel closed conformal real vector field is Ricci-flat (Calabi-Yau) and is flat in dimension 4.  Finally, we show that a Ricci soliton whose associated 1-form is harmonic, has constant scalar curvature. 
\end{abstract}

\maketitle

\section{Introduction}\label{int}

Gradient Ricci solitons are important  in understanding the Hamilton's Ricci flow \cite{Hamilton2}. They arise often as singularity models of the Ricci flow and that is why understanding them is an important question in the field. In view of their importance, it is then natural to seek classification results for gradient Ricci solitons. Indeed, the classification of gradient Ricci solitons has been a subject of interest for many people. We refer the reader to an excellent survey by Cao  \cite{caoALM11} and references therein for a nice overview on the subject.

A Riemannian metric $g$ on a smooth manifold $M^n$ is called {\it gradient Ricci soliton} if there exists a smooth potential function $f$ on $M^n$ such that the Ricci tensor ${\rm Ric}$ of the metric $g$ satisfies the following equation

\begin{equation}
\label{maineq}
{\rm Ric}+{\rm Hess}\,f=\lambda g,
\end{equation} for some constant $\lambda.$ Here, ${\rm Hess}\,f$ denotes the Hessian of $f.$  Clearly, when $f$ is a constant a gradient Ricci soliton is simply an Einstein manifold. If $\nabla f$ is replaced by a general vector field $V$, then the above equation defines a Ricci soliton. The Ricci soliton will be called {\it expanding}, {\it steady} or {\it shrinking} if  $\lambda<0,\,\lambda=0$ or $\lambda>0$, respectively. Ricci solitons are also of interest to physicists who refer to them as quasi-Einstein metrics (Friedan \cite{Friedan}).

A classical topic in this subject is to classify gradient Ricci solitons with constant scalar curvature. It is already known that compact Ricci solitons with constant scalar curvature are Einstein. Moreover, it is also easy to check that a steady gradient Ricci soliton with constant scalar curvature must be Ricci-flat. Petersen and Wylie \cite{PW2} showed that a shrinking (respectively, expanding) gradient Ricci soliton with constant scalar curvature satisfies $0\leq R\leq n\lambda$ (respectively, $n\lambda\leq R\leq 0$). Recently, Fern\'andez-L\'opez and Garc\'ia-R\'io \cite{MG} showed that if a gradient Ricci soliton $(M^{n},\,g)$ has constant scalar curvature $R,$ then $R\in\{0,\,\lambda,\,\ldots,\,(n-1)\lambda,\,n\lambda\}.$ They also showed the rigidity of gradient Ricci solitons with constant scalar curvature under the assumption that the Ricci operator has constant rank. Despite the progress, the complete classification of gradient Ricci solitons with constant scalar curvature still remains open. 

In order to make our approach more understandable, we need to recall some terminology. A smooth vector field $X$ on a Riemannian manifold $(M^n,g)$ is conformal if $\mathcal{L}_X g=2\psi g$ for some smooth function $\psi\in C^\infty(M),$ where $\mathcal{L}_X g$ is the Lie derivative of $g$ with respect to $X.$ In this case, the function $\psi$ is called conformal factor of $X.$ A particular case of a conformal vector $X$ is that for which 
\begin{equation}
\label{closedEq}
\nabla_{Y} X=\psi Y,\,\,\hbox{for any}\,\,Y\in\mathfrak{X}(M),
\end{equation} in this situation we say that $X$ is {\it closed}. They are also known as concircular vector fields (Fialkow \cite{Fialkow}) and appeared in the study of conformal mappings preserving geodesic circles having interesting applications in general relativity , e.g. trajectories of timelike concircular fields in the de Sitter model determine the world lines of receding or colliding galaxies satisfying the Weyl hypothesis (cf. \cite{Takeno}). Chen \cite{Chen} provided a simple characterization of generalized Robertson-Walker spacetimes in terms of timelike concircular vector field. Yet more particularly, a closed and conformal vector field $X$ is said to be {\it parallel} if $\psi$ vanishes identically. For more details, we refer the reader to  \cite{Amur/Hegde,Caminha,Obata,Obata/Yanno,Tanno/Webber} and \cite{Tashiro}.

We recall that the Gaussian Ricci soliton is given by $\Bbb{R}^n$ with canonical metric  and the potential function $f(x) = \frac{\lambda}{2}|x|^2,$ where $\lambda$ is an arbitrary constant. It has constant scalar curvature and carries a closed conformal vector field. In \cite{DRS}, Di\'ogenes, Ribeiro Jr. and Silva Filho were able to show that a complete gradient Ricci soliton carrying  a non-parallel closed conformal vector field must be either locally conformally flat for $n=3$ or $n=4,$ or has harmonic Weyl tensor for $n\ge 5.$ But, up to our knowledge, there is no classification results for steady and expanding gradient Ricci solitons with harmonic weyl tensor; the shrinking case was classified by Cao, Wang and Zhang in \cite{CWZ}, by  Fern\'{a}ndez-L\'opes  and Garc\'ia-R\'io in \cite{FLGR} and by Munteanu and Sesum \cite{MS}. It was shown by Jauregui and Wylie \cite{JW} that a gradient Ricci soliton admitting a non-homothetic conformal vector field $V$ that preserves the gradient $1$-form $df$ (i.e. $\nabla_{V} f$ is constant) is Einstein. Recently, Sharma \cite{Sharma} showed that gradient Ricci soliton with constant scalar curvature and admitting a non-homothetic conformal vector field leaving the potential vector field invariant is Einstein.

Motivated by these historical developments, \textit{we provide a classification result for gradient Ricci solitons with constant scalar curvature and carrying a non-parallel closed conformal vector field}. More precisely, we establish the following result.

\begin{theorem}\label{thmA}
Let $\big(M^n,\,g,\, f,\,\lambda\big),$ $n\geq3,$ be a complete gradient Ricci soliton with constant scalar curvature and carrying a  non-parallel closed conformal vector field. Then, $(M^n,\,g)$ is isometric to either 
\begin{enumerate}
\item the Euclidean space  $\Bbb{R}^{n}$, or
\item a round sphere $\Bbb{S}^n$, or
\item a negatively Einstein warped product of the real line with a complete non-positively Einstein manifold. 
\end{enumerate}
In dimension 4, the class (3) is a quotient of a hyperbolic space form $\Bbb{H}^n$. 
\end{theorem}

We also obtain the following corollaries.

\begin{corollary}  
If a complete gradient Ricci soliton $\big(M^n,\,g,\, f, \lambda\big),$ $n\geq3,$ with constant scalar curvature admits a non-homothetic closed conformal vector field of bounded norm, then it is isometric to a round sphere. 
\end{corollary}

\begin{corollary} A complete gradient shrinking Ricci soliton $\big(M^n,\,g,\, f, \lambda\big),$ $n\geq3,$ with constant scalar curvature and admitting a non-parallel closed conformal vector field is isometric to either a Euclidean space or a round sphere.
\end{corollary}

A key ingredient in establishing the proof of Theorem \ref{thmA} is a collection of identities pointed out by Ros and Urbano in \cite[Lemma 1]{RosUrbano}; see also \cite{Castro/Montealegre/Urbano}. It is worthwhile to remark that Theorem \ref{thmA} does not require any sign for the Ricci soliton constant $\lambda.$ At the same time, it does not require any relation between the gradient of the potential function and the assumed closed conformal vector field. Thus, our assumption is clearly weaker than the previous ones used in \cite{DRS,JW} and \cite{Sharma}. Besides, when $M^n$ is isometric to Euclidean space $\Bbb{R}^{n}$ in Theorem \ref{thmA}, it follows from Theorem 1.3 of Pigola, Rigoli, Rimoldi and Setti \cite{PRRS} that $f(x)=\frac{\lambda}{2}|x|^{2}+\langle b,\,x\rangle+c,$ for some $b\in\Bbb{R}^n$ and $c\in \Bbb{R}.$ 

\begin{remark}
The Bryant soliton admits closed conformal vector field, but its scalar curvature is not constant.  Hence, the existence of a closed conformal vector field on a gradient Ricci soliton does not imply the constancy of the scalar curvature. 
\end{remark}

\section{Background}

In this section we review some basic facts that will be useful in the proof of the main result. We begin by recalling some important features of gradient Ricci solitons (cf. Hamilton \cite{Hamilton2}).

\begin{lemma}
\label{lem1}
Let $\big(M^n,\,g,\,f,\,\lambda\big)$ be a gradient Ricci soliton. Then we have:
\begin{enumerate}
\item $R+\Delta f=n\lambda.$
\item $\frac{1}{2}\nabla R={\rm Ric}(\nabla f).$
\item $\Delta R=\langle \nabla R,\nabla f\rangle +2\lambda R-2|Ric|^{2}.$
\item $R+|\nabla f|^{2}=2\lambda f$ (after a possible rescaling).
\end{enumerate}
\end{lemma}

It follows from Lemma \ref{lem1} (1) and the Hopf maximum principle that a compact Ricci soliton with constant scalar curvature is necessarily Einstein. Moreover, cons-tant scalar curvature is a very restrictive condition for steady gradient Ricci soliton. Indeed, by Lemma \ref{lem1} (3), steady gradient Ricci solitons with constant scalar curvature are Ricci-flat. The same conclusion holds for gradient Ricci solitons with scalar flat curvature. 

According to Petersen and Wylie \cite{PW1}, a gradient Ricci soliton is said to be {\it rigid} if it is isometric to a quotient of $N\times \Bbb{R}^k,$ where $N$ is an Einstein manifold and $f=\frac{\lambda}{2}|x|^2$ on the Euclidean factor. Any rigid gradient Ricci soliton has constant scalar curvature. A special family of manifolds with constant scalar curvature are the homogeneous ones. In this context, it is known that any homogeneous gradient Ricci soliton is rigid (see \cite{PW1}). However, as it was previously mentioned, the complete classification of gradient Ricci solitons with constant scalar curvature still remains open.

The next lemma summarizes some useful well-known results on the theory of closed conformal vector fields (cf. Ross and Urbano \cite[Lemma 1]{RosUrbano}, see also Castro, Montealegre and Urbano \cite{Castro/Montealegre/Urbano}).

\begin{lemma}
\label{lemConf}
Let $(M^n,\,g)$ be a Riemannian manifold and $X$ a closed conformal vector field. Then the following assertions hold:
\begin{itemize}
\item[(1)] The set $\mathcal{Z}(X)$ of zeros of $X$ is a discrete set.
\item[(2)] $\nabla|X|^2=2\psi X.$
\item[(3)] $\nabla^2|X|^2=2\psi^2g+2d\psi\otimes X^\flat.$
\item[(4)] $|X|^{2}\nabla \psi=X(\psi) X.$
\item[(5)] The Ricci tensor satisfies ${\rm Ric}(X)=-(n-1)\nabla \psi .$
\end{itemize}
\end{lemma}

In \cite{joao}, Silva showed that every Ricci soliton endowed with a non-parallel homothetic closed vector field is a gradient Ricci soliton and has zero scalar curvature. Hence, by using Lemma \ref{lem1} (3) we obtain the following result.\\

\begin{proposition}
\label{propA}
Every Ricci soliton (not necessarily gradient) carrying a non-parallel homothetic closed conformal vector field is Ricci flat.
\end{proposition}

\section{Proof of Theorem1}
Throughout this section we will present the proof of Theorem \ref{thmA}. To begin with, we are going to obtain the following characterization of gradient Ricci solitons carrying a non-parallel closed conformal vector field; see also \cite{Silva1}. 

\begin{lemma}\label{lemA}
Let $(M^n, g, \nabla f, \lambda)$ be a gradient Ricci soliton admitting a non-parallel closed conformal vector field $X.$ Then one of the following assertions holds:
\begin{enumerate}
\item The gradient of $f$ satisfies $|X|^2\nabla f = X(f)X.$
\item $(M^n, g)$ is isometric to the Euclidean space $\Bbb{R}^n$.
\end{enumerate}
\end{lemma}

\begin{proof}
First, we claim that $$T = |X|^2d\psi \otimes df$$ is a symmetric $2$-tensor. Indeed, for an arbitrary vector field $Y$ on $M^n$ we have
\begin{eqnarray*}
YX(f) &=& \langle \nabla_Y \nabla f, X \rangle + \langle \nabla f, \nabla_Y X\rangle \\ 
&=& Hess\,f(X, Y) + \psi\langle \nabla f, Y\rangle\\
&=& \lambda \langle X, Y\rangle + (n-1)\langle \nabla \psi, Y\rangle 
+ \psi\langle \nabla f, Y\rangle,
\end{eqnarray*} where we used the fundamental equation (\ref{maineq}) jointly with Lemma \ref{lemConf} (5). The last equation shows that
\begin{eqnarray*}
\nabla X(f) = \lambda X + (n-1) \nabla \psi + \psi\nabla f,
\end{eqnarray*}
Differentiating it gives
\begin{eqnarray*}
\hspace{0,5cm} Hess\,X(f) = \lambda\psi g + (n-1)Hess\,\psi + \psi Hess\, f + d\psi\otimes df,
\end{eqnarray*} which immediately shows that $T$ is a symmetric $2$-tensor.

Using the fourth item of Lemma \ref{lemConf}, one can re-write $T$ as
\begin{eqnarray*}
T = X(\psi)X^{\flat} \otimes df.
\end{eqnarray*} Moreover, taking into account that $T$ is symmetric we infer
\begin{eqnarray*}
X(\psi)\Big(|X|^2\langle \nabla f, Y\rangle - X(f)\langle X, Y \rangle\Big) = 0,
\end{eqnarray*} which reduces to
\begin{eqnarray*}
X(\psi)\Big(|X|^2\nabla f - X(f) X\Big) = 0.
\end{eqnarray*} Next, recalling that  every gradient Ricci soliton has analytic metric (Ivey \cite[Theorem 1]{Ivey}) we conclude that either
\begin{eqnarray}\label{P1.1}
X(\psi) = 0
\end{eqnarray} 
or
\begin{eqnarray}\label{P1.2}
|X|^2\nabla f = X(f)X,
\end{eqnarray} where (\ref{P1.2}) is exactly the first assertion.

Now we must pay attention to (\ref{P1.1}). Indeed, in conjunction with Lemma \ref{lemConf} (4), it implies
\begin{eqnarray*}
|X|^2\nabla\psi = 0,
\end{eqnarray*} 
and consequently
\begin{eqnarray*}
\nabla\psi = 0,
\end{eqnarray*} for all non-singular points of $X.$ Thus, from Lemma \ref{lemConf} (1) we conclude, by continuity, that the last identity holds everywhere and therefore, $\psi$ is constant.

Finally, taking into account that $X$ is non-parallel, we define the function $h: M^n \rightarrow \Bbb{R}$ by 
\begin{eqnarray*}
h = \frac{1}{2\psi}|X|^2.
\end{eqnarray*}
A straightforward computation using this function, in conjunction with Lemma \ref{lemConf} (3) shows that
\begin{eqnarray*}
Hess\, h = \psi g.
\end{eqnarray*}
Hence, we invoke Theorem 2 of \cite{Tashiro} to conclude that $(M^n, g)$ is isometric to Euclidean space, completing the proof of the lemma.
\end{proof}

\vspace{0.4cm}

Now we are ready to conclude the proof of Theorem \ref{thmA}. First, we suppose that $M^n$ is a gradient Ricci soliton carrying a non-parallel closed conformal vector field $X$. Then, it follows by Lemma \ref{lemA} that 
\begin{equation}
\label{eq1p}
|X|^{2}\nabla f=X(f)X,
\end{equation}
or ($M^n,g$) is isometric to the Euclidean space $\Bbb{R}^n$. Hereafter, since $M^n$ has constant scalar curvature we  use Lemma \ref{lem1} (2) to infer

$$Ric(|X|^{2}\nabla f)=|X|^{2}Ric(\nabla f)=\frac{1}{2}|X|^{2}\nabla R=0.$$ Using this consequence in (\ref{eq1p}) provides $X(f)Ric(X)=0$. Consequently, $$|X|^{2}X(f)Ric(X)=0.$$ Now, using Lemma \ref{lemConf} (5) we obtain $|X|^{2}X(f)\nabla \psi=0$ and then by Lemma \ref{lemConf} (4), it follows that $X(f)X(\psi)X=0.$ So, it follows from (\ref{eq1p}) that $$X(\psi)|X|^{2}\nabla f=0.$$  This implies that $X(\psi)\nabla f=0$ in $M\setminus\mathcal{Z}(X),$ where $\mathcal{Z}(X)$ is the set of zeros of $X.$ Taking into account that $\mathcal{Z}(X)$ is a discrete set, we conclude 
\begin{equation}
\label{eqAq}
X(\psi)\nabla f=0\,\,\,\hbox{on}\,\,\,M^n.
\end{equation} Indeed, (\ref{eqAq}) is trivially satisfied for the zeros of $X.$ 

Proceeding further, since any gradient Ricci soliton is analytic in harmonic coordinates, $\nabla f$ can not vanish on any nonempty open dense subset of $M^n$, otherwise, it will be Einstein and $f$ will be constant. Thus, we  use (\ref{eqAq}) to conclude that $X(\psi)=0$. Using  Lemma \ref{lemConf} (4) again, we deduce that $\nabla \psi=0,$ and this therefore forces $\psi$ to be constant. Now, we define the function $\varphi: M^n \rightarrow \Bbb{R}$ by
\begin{eqnarray*}
\varphi = \frac{1}{2\psi}|X|^2.
\end{eqnarray*} Computing its Hessian, and using Lemma \ref{lemConf} (3) we obtain
\begin{eqnarray*}
Hess\, \varphi = \psi g.
\end{eqnarray*}
Thus, it suffices to apply Theorem 2 of \cite{Tashiro} to conclude that $(M^n,\,g)$ is isometric to Euclidean space $\Bbb{R}^n.$

Finally, we turn our attention to the case when $M^n$ is Einstein   and $f$ is constant. By Theorem G of Kanai \cite{Kanai}, we know for the positively Einstein case that $M^n$ is isometric to a Euclidean sphere, and for the Ricci-flat case that it is isometric to the Euclidean space $\Bbb{R}^n$. For the negatively Einstein case, i.e. $R < 0$. Let us define $k$ by $R=n(n-1)k$. Then $k < 0$. By virtue of part (ii) of Theorem G of \cite{Kanai}, we find that ($M^n,g$) is isometric to the warped product $(\bar{M},\bar{g})_{\xi}\times (R,g_0)$ of a complete Einstein manifold $(\bar{M},\bar{g})$ of constant scalar curvature 4$(n-1)(n-2)kc_1 c_2$ and the real line $(R,g_0)$, warped by $\xi(t)=c_1 e^{\sqrt{-k}t}+c_2 e^{-\sqrt{-k}t}$ such that $c_1$ and $c_2$ are non-negative constants. In particular, for dimension 4, the fiber $\bar{M}$ is 3-dimensional, and hence has zero Weyl tensor. By Gauss and Codazzi equations, it follows by a straightforward calculation that $M$ is conformally flat, and as it is negatively Einstein, it has constant negative curvature. This completes the proof.

\subsection{Proof of Corollary 1} 
\begin{proof} As $X$ is non-homothetic, the case (1)  of Theorem 1 are ruled out. So, it would suffice to rule out the case (3). Let us decompose the closed conformal vector field $X$ orthogonally as $\bar{X}+\alpha T$ where $\bar{X}$ is the component along $\bar{M}$ and $T$ is the lift of $\frac{d}{dt}$ on $M^n$. The components $\nabla_{T}X=\psi T$ of the closed conformal equation, the warped product formulas $\nabla_T \bar{X}=\frac{\dot \xi}{\xi}\bar{X}$ and $\nabla_T T = 0$, and comparison of components tangent to $\bar{M}$ and $R$ lead to $\dot{\alpha}=\psi$ and $\dot \xi \bar{X}= 0$. So, either $\bar{X}=0$ or $\dot{\xi}=0$ on an open sub-interval of $R$, which is not possible because $\xi(t)=c_1 e^{\sqrt{-k}t}+c_2 e^{-\sqrt{-k}t}$. Hence $\bar{X}=0$ and thus $X = \alpha T$. Now, using $\nabla_{\bar{Y}}X = \psi \bar{Y}$, where $\bar{Y}$ is an arbitrary vector field tangent to $\bar{M}$ and comparing components along $\bar{M}$ and $R$ shows that $\alpha$ is a function of $t$ alone, and $\psi = \alpha \frac{\dot{\xi}}{\xi}$. Consequently, we find $\dot{\alpha}=\alpha \frac{\dot{\xi}}{\xi}$ which easily integrates as $\alpha = \xi$ (the constant factor due to integration can be absorbed by the constants in $\xi$). Therefore, we obtain $X = \xi T$, and so the norm of $X$ is $\xi$ which is unbounded and this contradicts our assumption that $X$ has bounded norm. This completes the proof.
\end{proof}

\subsection{Proof of Corollary 2} 
\begin{proof}If $\psi$ is a non-zero constant, then as shown in the proof of Theorem 1, $M^n$ is isometric to a Euclidean space. If $\psi$ is non-constant, then as shown in the proof of Theorem 1, $f$ is constant and $g$ is Einstein. The constancy of $R$ and Lemma \ref{lem1} (3) implies that $|Ric|^{2}=\lambda R$. Thus, the hypothesis $\lambda > 0$ implies that $R > 0$, because $R = 0$ would imply Ricci flatness and $\lambda =0$, contradicting our hypothesis. As $g$ is Einstein, Lemma \ref{lemConf} (5) provides $X = -\frac{n(n-1)}{R}\nabla \psi$, i.e. $X$ is gradient. Differentiating the foregoing equation we obtain $Hess \ \psi = -\frac{R}{n(n-1)}g$. This implies, by Obata's theorem \cite{Obata}, that $(M^n, g)$ is isometric to the round sphere $\Bbb{S}^n$, completing the proof.
\end{proof}

\section{K\"ahler Gradient Ricci Soliton}

When the underlying manifold is a complex manifold, we have the corresponding notion of a K\"ahler-Ricci soliton. Recall that a complex $n$-dimensional (real $2n$-dimensional) Riemannian manifold $(M,\,g)$ with a complex structure $J: TM \rightarrow TM$ defined by $J^2 = -I$, is a K\"ahler manifold provided the metric $g$ is $J$-invariant (Hermitian), i.e. $g(JY,JZ) = g(Y,Z)$ for arbitrary real vector fields $Y,Z$, and $J$ is parallel with respect to the Riemannian connection $\nabla$ of $g$, i.e. $\nabla J = 0$. The metric $g$ is called the K\"ahler metric and its Ricci tensor is $J$-invariant, i.e. $Ric \circ J = J \circ Ric$.  Following Chow et al. (p. 97, \cite{Chow}), we state the following definition.
\begin{definition}A K\"ahler-Ricci soliton is a K\"ahler manifold ($M,g,J$) such that the soliton structure equation
\begin{equation}\label{T2.1}
\mathcal{L}_V g +2Ric=2\lambda g
\end{equation}
holds for a real constant $\lambda$ and a vector field $V$ that is an infinitesimal automorphism of the complex structure $J$, i.e. 
\begin{equation}\label{T2.2}
\mathcal{L}_V J = 0.
\end{equation}
\end{definition}
One may note here that the condition (\ref{T2.2}) is automatically satisfied for a gradient K\"ahler-Ricci soliton (for which $V = \nabla f$ for a smooth function $f$ on $M$). Now we state our next result as follows. 

\begin{proposition}
If $(M^{2n},\, g,\, \nabla f,\, \lambda),$ $2n\geq 4,$ is a gradient Kaehler Ricci soliton real dimension $2n>4$, with a non-parallel closed conformal real vector field $X$, then (i )$ X$ is homothetic and an infinitesimal automorphism of the complex structure, and (ii) it is Ricci-flat and is flat in dimension 4. 
\end{proposition}

Its proof uses the following known result(Goldberg \cite[p. 265]{Goldberg}).

\begin{lemma} A closed conformal vector field on a K\"ahler manifold of real dimension $2n \ge 4$ is homothetic and an infinitesimal automorphism of the complex structure.
\end{lemma}

\textbf{Proof of the Proposition.}
As $X$ is homothetic, we invoke Proposition \ref{propA} to conclude that $Ric = 0.$ Hence, it follows from (\ref{closedEq}) that the curvature tensor annihilates $X,$ and therefore, the Weyl conformal tensor $W$ annihilates $X.$ So, if the real dimension of $M$ is $4,$ then a combinatorial computation using the symmetries and complete traceless property of $W$ implies that $W = 0$ at all points other than the zeros of $X.$ But the zeros of $X$ are discrete, and hence using the continuity of $W$ we conclude that $W = 0$ on $M$. But $Ric = 0$, and so $M$ is flat.  This completes the proof.

\begin{remark}A Ricci flat K\"ahler manifold (for which the first Chern class vanishes) is known as a Calabi-Yau manifold which has 
application in superstring theory based on a $10$-dimensional manifold that is the product of the $4$-dimensional space-time
and a $6$-dimensional Calabi-Yau manifold (see Candelas et al. \cite{Candelas}).
\end{remark}
\section{A constant scalar curvature condition for Ricci soliton}
Let ($M,g,V,\lambda$) be a Ricci soliton (not necessarily gradient). Then
\begin{equation}\label{5.1}
(\mathcal{L}_V g)(Y,Z)+2Ric(Y,Z)=2\lambda g(Y,Z)
\end{equation}
for arbitrary $Y,Z \in \mathfrak{X}(M)$. Decomposing $g(\nabla_Y V,Z)$ into symmetric and skew-symmetric parts, and using (\ref{5.1}) we have
\begin{equation}\label{5.2}  
g(\nabla_Y V,Z)=-Ric(Y,Z)+\lambda g(Y,Z) +(dv)(Y,Z)
\end{equation}
where $dv$ is the exterior derivative of the 1-form $v$ metrically equivalent to $V$, i.e. $v(Y)=g(V,Y)$. Let us set $(dv)(Y,Z)=g(FX,Z)$. Then $F$ is a skew-symmetric (1,1)-tensor field, i.e. $g(FY,Z)=-g(Y,FZ)$. Factoring $Z$ out of (\ref{5.2}) gives
\begin{equation}\label{5.3}
\nabla_Y V=-Ric(Y)+\lambda Y +FY
\end{equation}
Tracing it provides
\begin{equation}\label{5.4}
\delta v = -div(V)=R-n\lambda
\end{equation}
where $\delta$ is the co-differential operator. We establish the following result that provides a condition for a Ricci soliton to have constant scalar curvature.

\begin{theorem} Let ($M,g,V,\lambda$) be a Ricci soliton. If the 1-form $v$ metrically equivalent to $V$ is harmonic, then $M$ has constant scalar curvature, and the Ricci operaror annihilates $V$.
\end{theorem}
\begin{proof}As $v$ is harmonic,
\begin{equation}\label{5.5}
 \Delta v=d\delta v +\delta dv = 0,
\end{equation}
where $\Delta$ is the Hodge Laplacian operator. Also, $(dv)(Y,Z)=g(FY,Z)$ shows that $\delta dv=div(F)$. This, in conjunction with equation (\ref{5.4}) shows that  equation (\ref{5.5}) assumes the form

\begin{equation}\label{5.6}
dR+div(F)=0.
\end{equation}
A straightforward computation using (\ref{5.3})  leads to

\begin{equation}\label{5.7}
Ric(V)=\frac{1}{2}\nabla R+(div(F))^{*}
\end{equation}
where $(div(F))^{*}$ denotes the vector field metrically equivalent to the 1-form $div(F)$. At this point, we recall the following result (Stepanov and Shelepova \cite{S-S}): If $(M,g,V,\lambda)$ is a Ricci soliton, then $\square V=0$, where $\square$ is the Yano operator transformiong a vector field $V$ into the vector field with components $(\square V)^i =-(\nabla^j \nabla_j V^i+R^i _j V^j)$. By our hypothesis, $\Delta v=0$, which is equivalent to  (Yano \cite{Yano})
\begin{equation*}
-(\nabla^j \nabla_j V^i-R^i _j V^j)=0,
\end{equation*}
and hence, in conjunction with the aforementioned result yields $Ric(V)=0$. The use of this consequence in equations (\ref{5.6}) and (\ref{5.7}) immediately shows that $R$ is constant, completing the proof.

\end{proof}

\begin{remark} For a gradient Ricci soliton $v =df$, we note that the constant scalar curvare condition is easily seen to be equivalent to $\Delta v=0$, and also equivalent to $Ric(V)=0$, in view of Lemma \ref{lem1} (2).
\end{remark}

The authors have no conflicts of interest to declare that are relevant to the content of this article.

\end{document}